\documentclass[12pt]{amsart}

\usepackage{amsmath,amssymb,amsfonts,amsxtra,hyperref,fullpage}
\usepackage[pagewise]{lineno}
\newtheorem{theorem}{Theorem}[subsection]
\newtheorem{lemma}[theorem]{Lemma}

\newtheorem{corollary}[theorem]{Corollary}

\newtheorem{question}[theorem]{Question}
\newtheorem{remark}[theorem]{Remark}

\DeclareMathOperator{\bbE}{\mathbb{E}}
\DeclareMathOperator{\bbN}{\mathbb{N}}

\DeclareMathOperator{\A}{\mathcal{A}}
\DeclareMathOperator{\Aut}{\operatorname{Aut}}
\DeclareMathOperator{\B}{\mathcal{B}}
\DeclareMathOperator{\C}{\mathcal{C}}

\DeclareMathOperator{\Fix}{\operatorname{Fix}}

\DeclareMathOperator{\id}{\operatorname{id}}
\DeclareMathOperator{\J}{\mathcal{J}}

\DeclareMathOperator{\LIM}{\operatorname{LIM}}
\DeclareMathOperator{\LP}{\operatorname{LP}}

\let\slasho=\o
\renewcommand{\o}{\overline}

\DeclareMathOperator{\RP}{\operatorname{RP}}
\renewcommand{\S}{\mathcal{S}}

\begin{document}

\title[Unique Expectations]{Unique Expectations for\\ Discrete Crossed Products}
\author{Vrej Zarikian}
\address{U. S. Naval Academy, Annapolis, MD 21402}
\email{zarikian@usna.edu}
\begin{abstract}
Let $G$ be a discrete group acting on a unital $C^*$-algebra $\A$ by $*$-automorphisms. We characterize (in terms of the dynamics) when the inclusion $\A \subseteq \A \rtimes_r G$ has a unique conditional expectation, and when it has a unique pseudo-expectation (in the sense of Pitts). Likewise for the inclusion $\A \subseteq \A \rtimes G$. As an application, we (slightly) strengthen results of Kishimoto and Archbold-Spielberg concerning $C^*$-simplicity of $\A \rtimes_r G$.
\end{abstract}
\maketitle

\section{Introduction}

Let $\B$ be a unital $C^*$-algebra and $\A \subseteq \B$ be a unital $C^*$-subalgebra, with $1_{\A} = 1_{\B}$. In short, let $\A \subseteq \B$ be a \emph{$C^*$-inclusion}. Recently we have been concerned with characterizing when a $C^*$-inclusion admits a unique conditional expectation, or a unique pseudo-expectation (in the sense of Pitts), because significant structural consequences often ensue in both cases \cite{PittsZarikian2015,Zarikian2017}. This paper continues the program, with $\B$ equal to the crossed product of $\A$ by a discrete group $G$.\\

A \emph{conditional expectation} for a $C^*$-inclusion $\A \subseteq \B$ is a unital completely positive (ucp) map $E:\B \to \A$ such that $E|_{\A} = \id_{\A}$. Conditional expectations are automatically $\A$-bimodular, so that $E(ax) = aE(x)$ and $E(xa) = E(x)a$ whenever $x \in \B$ and $a \in \A$. Unfortunately, a $C^*$-inclusion often admits no conditional expectations at all.\\

In \cite{Pitts2017}, Pitts introduced pseudo-expectations as a substitute for possibly non-existent conditional expectations. A \emph{pseudo-expectation} for a $C^*$-inclusion $\A \subseteq \B$ is a ucp map $\theta:\B \to I(\A)$ such that $\theta|_{\A} = \id_{\A}$. Here $I(\A)$ is Hamana's \emph{injective envelope} of $\A$ (discussed in detail below). Every conditional expectation is a pseudo-expectation, but the converse is false. Just like conditional expectations, pseudo-expectations are $\A$-bimodular. Unlike conditional expectations, pseudo-expectations need not be idempotent. Indeed, if $\theta:\B \to I(\A)$ is a pseudo-expectation for $\A \subseteq \B$, then the composition $\theta \circ \theta$ is typically undefined, since it is rarely the case that $I(\A) \subseteq \B$. Furthermore, pseudo-expectations are difficult to describe explicitly, since $I(\A)$ only admits a concrete description in exceptional situations.\\

In spite of their drawbacks, pseudo-expectations enjoy two tremendous technical advantages over conditional expectations, both related to the fact that $I(\A)$ is injective. First, pseudo-expectations always exist for any $C^*$-inclusion $\A \subseteq \B$. Indeed, the identity map $\id_{\A}:\A \to \A$ always has a ucp extension $\theta:\B \to I(\A)$, by injectivity. Second, and more generally, pseudo-expectations always extend. That is, if $\theta:\B \to I(\A)$ is a pseudo-expectation for $\A \subseteq \B$, and if $\B \subseteq \C$, then there exists a pseudo-expectation $\tilde{\theta}:\C \to I(\A)$ for $\A \subseteq \C$, such that $\tilde{\theta}|_{\B} = \theta$.\\

In our experience, for the reasons detailed above, it is easier to characterize when a $C^*$-inclusion admits a unique pseudo-expectation, than to characterize when it admits a unique (or at most one) conditional expectation. Of course, if a $C^*$-inclusion admits a unique pseudo-expectation, then it admits at most one conditional expectation. So it can be profitable to consider pseudo-expectations, even if one is ultimately interested in conditional expectations. Moreover, because having a unique pseudo-expectation is a stronger condition than having at most one conditional expectation, it usually imposes tougher structural constraints on the inclusion.\\

In \cite{PittsZarikian2015}, we investigate the unique pseudo-expectation property for $C^*$-inclusions, pursuing two directions. On the one hand, we relate the unique pseudo-expectation property to other structural properties of the inclusion. For example, we show that if a $C^*$-inclusion admits a unique pseudo-expectation which is faithful, then the inclusion is hereditarily essential \cite[Thm. 3.5]{PittsZarikian2015}. (A $C^*$-inclusion $\A \subseteq \B$ is \emph{essential} if every non-trivial ideal $\J \subseteq \B$ intersects $\A$ non-trivially. It is \emph{hereditarily essential} if the $C^*$-inclusion $\A \subseteq \B_0$ is essential, for every intermediate $C^*$-algebra $\A \subseteq \B_0 \subseteq \B$.) On the other hand, we characterize when various special classes of $C^*$-inclusions admit a unique pseudo-expectation. In particular, we show that if $(\A,G,\alpha)$ is a $C^*$-dynamical system with $\A$ abelian and $G$ discrete, then the inclusion $\A \subseteq \A \rtimes_r G$ (reduced crossed product) admits a unique pseudo-expectation (necessarily a faithful conditional expectation) if and only if the induced action of $G$ on $\widehat{\A}$ is \emph{topologically free} \cite[Thm. 4.6]{PittsZarikian2015}.\\

In this paper, we substantially generalize the aforementioned \cite[Thm. 4.6]{PittsZarikian2015}. For $C^*$-dynamical systems $(\A,G,\alpha)$ with $\A$ arbitrary and $G$ discrete, we characterize (in terms of the dynamics) when $\A \subseteq \A \rtimes_r G$ admits a unique pseudo-expectation, as well as when it admits a unique conditional expectation. There is a unique pseudo-expectation if and only if the action of $G$ is \emph{properly outer} (Theorem \ref{pseudo}), and there is a unique conditional expectation if and only if $G$ \emph{acts freely} (Theorem \ref{CE}). The same statements hold for the inclusion $\A \subseteq \A \rtimes G$ (full crossed product). If the action of $G$ is properly outer, then $G$ acts freely, but the converse is false. Thus we can systematically produce $C^*$-inclusions with a unique conditional expectation, but multiple pseudo-expectations. (The first such example appears in \cite{Zarikian2017}.) Additionally, by combining Theorem \ref{pseudo} with the aforementioned \cite[Thm. 3.5]{PittsZarikian2015}, we quickly reprove (and slightly strengthen) $C^*$-simplicity results for reduced crossed products, originally due to Kishimoto \cite{Kishimoto1981} and Archbold-Spielberg \cite{ArchboldSpielberg1993}.\\

We hope that the results of this paper motivate other researchers, especially $C^*$-specialists, to take up the study of pseudo-expectations.

\section{Preliminaries}

\subsection{Discrete Crossed Products}

Let $\A$ be a unital $C^*$-algebra, $G$ be a discrete group, and $\alpha:G \to \Aut(\A)$ be a homomorphism. In short, let $(\A,G,\alpha)$ be a (discrete) $C^*$-dynamical system. We denote by $\A \rtimes_r G$ (resp. $\A \rtimes G$) the reduced (resp. full) crossed product of $\A$ by $G$ with respect to $\alpha$. That is, $\A \rtimes_r G$ is the completion of the $\alpha$-twisted convolution algebra $C_c(G,\A)$ with respect to the norm induced by the regular representation, while $\A \rtimes G$ is the completion of $C_c(G,\A)$ with respect to the norm induced by the universal representation. Evidently, there exists a unital $*$-homomorphism $\lambda:\A \rtimes G \to \A \rtimes_r G$ which fixes $C_c(G,\A)$. There is also a faithful conditional expectation $\bbE:\A \rtimes_r G \to \A$ such that
\[
    \bbE(g) = \begin{cases} 1, & g = e\\ 0, & g \neq e \end{cases}, ~ g \in G.
\]
This, in turn, gives rise to a canonical conditional expectation $\tilde{\bbE} = \bbE \circ \lambda:\A \rtimes G \to \A$.

\subsection{Hamana's Injective Envelope}

For every unital $C^*$-algebra $\A$, there exists a minimal injective operator system $I(\A)$ containing $\A$, called the \emph{injective envelope} of $\A$ \cite{Hamana1979}. That is, $I(\A)$ is an injective operator system containing $\A$ as an operator subsystem, and if $\A \subseteq \S \subseteq I(\A)$ is an injective operator system, then $\S = I(\A)$. The minimality of $I(\A)$ is equivalent to the \emph{rigidity} of the inclusion $\A \subseteq I(\A)$. That is, if $\Phi:I(\A) \to I(\A)$ is a ucp map such that $\Phi|_{\A} = \id_{\A}$, then $\Phi = \id_{I(\A)}$. Using rigidity, it is easy to see that $I(\A)$ is uniquely determined up to a complete order isomorphism which fixes $\A$.\\

A priori, $I(\A)$ is just an operator system. However, it turns out that $I(\A)$ has a wealth of algebraic and analytical structure. It is a monotonically complete $C^*$-algebra (and thus an $AW^*$-algebra) containing $\A$ as a unital $C^*$-subalgebra. As such, it enjoys many of the nice features one normally associates with von Neumann algebras. In particular:
\begin{itemize}
\item The projections in $I(\A)$ form a complete lattice.
\item For every element $x \in I(\A)$, there exists a smallest projection $\LP(x) \in I(\A)$ such that $\LP(x)x = x$. Likewise, there exists a smallest projection $\RP(x) \in I(\A)$ such that $x\RP(x) = x$.
\item For every element $x \in I(\A)$, there exists a partial isometry $v \in I(\A)$ such that $x = v|x|$, $vv^* = \LP(x)$, and $v^*v = \RP(x)$.
\end{itemize}
It is not true in general that $I(\A)$ is a dual Banach space, and so weak-$*$ convergence does not make sense in $I(\A)$. On the other hand, there is a well-behaved mode of convergence which often plays the same role. We say that $x \in I(\A)$ is the \emph{order limit} of a net $\{x_j\} \subseteq I(\A)$, and write $x = \LIM_j x_j$, provided there are increasing nets $\{a_j\}, \{b_j\}, \{c_j\}, \{d_j\} \subseteq I(\A)_{sa}$ with suprema $a, b, c, d \in I(\A)_{sa}$, respectively, such that $x_j = (a_j - b_j) + i(c_j - d_j)$ for all $j$ and $x = (a - b) + i(c - d)$. (It can be shown that this definition is independent of which increasing nets one uses.) See \cite[Ch. 2 \& 8]{SaitoWright2015} for a recent treatment of this material.

\subsection{Dynamics}

For a unital $C^*$-algebra $\A$, we denote by $\Aut(\A)$ the $*$-automorphisms of $\A$. Every $\alpha \in \Aut(\A)$ has a unique extension $\tilde{\alpha} \in \Aut(I(\A))$. Consequently, many dynamical properties of $\alpha$ can be rephrased in terms of dynamical properties of $\tilde{\alpha}$, where the situation is usually simpler. In particular:
\begin{itemize}
\item $\alpha$ is quasi-inner $\iff$ $\tilde{\alpha}$ is inner (\cite[Thm. 7.4]{Hamana1985});
\item $\alpha$ is properly outer $\iff$ $\tilde{\alpha}$ is properly outer $\iff$ $\tilde{\alpha}$ is freely acting (\cite[Prop. 5.1]{Hamana1982} and \cite[Rmk. 7.5]{Hamana1985}).
\end{itemize}
(We say that $\alpha$ is \emph{inner} if $\alpha(a) = uau^*$, $a \in \A$, where $u \in \A$ is unitary. We take the first equivalence above as the definition of \emph{quasi-innerness}, although it can also be defined using the Borchers spectrum (cf. \cite[p. 477]{Hamana1985}). We say that $\alpha$ is \emph{properly outer} if there does not exist a non-zero $\alpha$-invariant ideal $\J \subseteq \A$ such that $\alpha|_{\J}$ is quasi-inner. Equivalently, $\alpha$ is properly outer if there does not exist a non-zero $\tilde{\alpha}$-invariant projection $p \in I(\A)$ such that $\tilde{\alpha}|_{pI(\A)p}$ is inner. Finally, we say that $\alpha$ is \emph{freely acting} if it has no non-zero \emph{dependent elements}, i.e., the only $b \in \A$ such that $ba = \alpha(a)b$, $a \in \A$, is $b = 0$.)\\

For automorphisms of $I(\A)$, being properly outer and acting freely are equivalent, as indicated above. For automorphisms of $\A$, proper outerness is in general the stronger condition. Put another way, if $\tilde{\alpha}$ acts freely, then so does $\alpha$. Indeed, as implied by the following technical lemma, dependent elements for $\alpha$ are also dependent elements for $\tilde{\alpha}$. To construct a freely acting automorphism $\alpha$ such that $\tilde{\alpha}$ is inner, one can use \cite[Prop. 1]{Choda1974}.

\begin{lemma} \label{dependent}
Let $\A$ be a unital $C^*$-algebra, $\alpha \in \Aut(\A)$, and $x \in I(\A)$. If
\[
    xa = \alpha(a)x, ~ a \in \A,
\]
then
\[
    xt = \tilde{\alpha}(t)x, ~ t \in I(\A).
\]
\end{lemma}

\begin{proof}
We may assume that $\|x\| \leq 1$. Arguing as in \cite{ChodaKasaharaNakamoto1972}, we see that
\[
    x^*x = xx^* \in \A' \cap I(\A) = Z(I(\A)).
\]
Thus $|x| \in Z(I(\A))$. Let $v \in I(\A)$ be a partial isometry such that $x = v|x|$, $vv^* = \LP(x)$, and $v^*v = \RP(x)$. We have that $\LIM_n |x|^{1/n} = v^*v$. For all $a \in \A$,
\begin{eqnarray*}
    v|x|a = \alpha(a)v|x|
    &\implies& v|x|^na = \alpha(a)v|x|^n, ~ n \in \bbN\\
    &\implies& v|x|^{1/n}a = \alpha(a)v|x|^{1/n}, ~ n \in \bbN\\
    &\implies& vv^*va = \alpha(a)vv^*v\\
    &\implies& va = \alpha(a)v.
\end{eqnarray*}
Thus, as before,
\[
    v^*v = vv^* \in Z(I(\A)).
\]
Set $p = v^*v$, a projection in $Z(I(\A))$, and define a ucp map $\theta:I(\A) \to I(\A)$ by the formula
\[
    \theta(t) = v^*\tilde{\alpha}(t)v + p^\perp t, ~ t \in I(\A).
\]
For all $a \in \A$, we have that
\[
    \theta(a) = v^*\alpha(a)v + p^\perp a = v^*va + p^\perp a = pa + p^\perp a = a.
\]
By rigidity, $\theta = \id_{I(\A)}$, and so
\[
    v^*vt = v^*\tilde{\alpha}(t)v, ~ t \in I(\A).
\]
Pre-multiplying by $v$ yields
\[
    vt = vv^*\tilde{\alpha}(t)v = \tilde{\alpha}(t)v, ~ t \in I(\A).
\]
It follows that
\[
    xt = \tilde{\alpha}(t)x, ~ t \in I(\A),
\]
as desired.
\end{proof}

\section{Unique Expectations}

\subsection{Unique Conditional Expectations}

In this section we show that $\A \subseteq \A \rtimes_r G$ admits a unique conditional expectation if and only if $G$ acts freely on $\A$. We begin with a lemma of independent interest, which was inspired by \cite[Prop. 3.1.4]{Stormer2013}.

\begin{lemma} \label{multdomain}
Let $\A \subseteq \B$ be a $C^*$-inclusion. Assume that there exists a unique conditional expectation $E:\B \to \A$. Then $E$ is multiplicative on $\A^c = \A' \cap \B$. If, in addition, $E$ is faithful, then $\A^c = Z(\A)$.
\end{lemma}

\begin{proof}
Since $E$ is $\A$-bimodular, $E(\A^c) = Z(\A)$. Let $x \in (\A^c)_{sa}$, with $\|x\| < 1$. Then $1-x$ is a positive invertible element of $\A^c$, and $1-E(x)$ is a positive invertible element of $Z(\A)$. Define a ucp map $\theta:\B \to \A$ by the formula
\[
    \theta(t) = E((1-x)^{1/2}t(1-x)^{1/2})(1-E(x))^{-1}, ~ t \in \B.
\]
It is easy to see that $\theta(a) = a$, $a \in \A$, so that $\theta$ is a conditional expectation. By assumption, $\theta = E$, and so
\[
    E(x)(1-E(x)) = E((1-x)^{1/2}x(1-x)^{1/2}),
\]
which implies $E(x^2) = E(x)^2$. It follows that $x$ is in the multiplicative domain of $E$. Since the choice of $x$ was arbitrary, $E|_{\A^c}:\A^c \to Z(\A)$ is a $*$-homomorphism. If $E$ is faithful, then $E|_{\A^c}$ is injective. In that case, $x = E(x) \in Z(\A)$ for all $x \in \A^c$, since $E(x-E(x)) = 0$.
\end{proof}

\begin{theorem} \label{CE}
Let $(\A,G,\alpha)$ be a discrete $C^*$-dynamical system. Then the following are equivalent:
\begin{enumerate}
\item[i.] $\A \subseteq \A \rtimes_r G$ admits a unique conditional expectation;
\item[ii.] $\A^c = Z(\A)$;
\item[iii.] $G$ acts freely on $\A$.
\end{enumerate}
\end{theorem}

\begin{proof}
(i $\implies$ ii) Lemma \ref{multdomain}.

(ii $\implies$ iii) Suppose $\A^c = Z(\A)$. Let $e \neq g \in G$ and $b \in \A$, and assume that $ba = \alpha_g(a)b$ for all $a \in \A$. Then $g^{-1}b \in \A^c$, which implies $b = 0$.

(iii $\implies$ i) Suppose $G$ acts freely on $\A$. Let $\theta:\A \rtimes_r G \to \A$ be a conditional expectation. Fix $e \neq g \in G$. For all $a \in \A$, $gag^{-1} = \alpha_g(a)$, which implies $ga = \alpha_g(a)g$, which in turn implies
\[
    \theta(g)a = \theta(ga) = \theta(\alpha_g(a)g) = \alpha_g(a)\theta(g).
\]
It follows that $\theta(g) = 0$. Since the choice of $g$ was arbitrary, $\theta = \bbE$.
\end{proof}

\begin{corollary} \label{CE_full}
Let $(\A,G,\alpha)$ be a discrete $C^*$-dynamical system. Then $\A \subseteq \A \rtimes G$ (full crossed product) admits a unique conditional expectation if and only if $G$ acts freely on $\A$.
\end{corollary}

\begin{proof}
($\Rightarrow$) Let $\theta:\A \rtimes_r G \to \A$ be a conditional expectation. Then $\theta \circ \lambda:\A \rtimes G \to \A$ is a conditional expectation, so that $\theta \circ \lambda = \bbE \circ \lambda$, by uniqueness. Thus $\theta = \bbE$. By Theorem \ref{CE}, $G$ acts freely on $\A$.

($\Leftarrow$) Conversely, suppose $G$ acts freely on $\A$. Let $\Theta:\A \rtimes G \to \A$ be a conditional expectation. Then repeating the proof of (iii $\implies$ i) in Theorem \ref{CE} above, with $\theta$ replaced by $\Theta$, we see that $\Theta(g) = 0$ for all $g \neq e$. Hence $\Theta = \tilde{\bbE}$.
\end{proof}

\subsection{Unique Pseudo-Expectations}

In this section we show that $\A \subseteq \A \rtimes_r G$ (resp. $\A \subseteq \A \rtimes G$) admits a unique pseudo-expectation if and only if the action of $G$ on $\A$ is properly outer. We begin with a technical lemma, similar in spirit to \cite[Lemma 5.1.6]{EffrosRuan2000}.

\begin{lemma} \label{factorization}
Let $\A \subseteq \B$ be a $C^*$-inclusion and $\theta:\B \to I(\A)$ be a completely positive $\A$-bimodule map. Then there exists a ucp $\A$-bimodule map $\tilde{\theta}:\B \to I(\A)$ (i.e., a pseudo-expectation for $\A \subseteq \B$) such that $\theta(x) = \theta(1)\tilde{\theta}(x)$, $x \in \B$.
\end{lemma}

\begin{proof}
Fix any ucp $\A$-bimodule map $\Phi:\B \to I(\A)$ (i.e., any ucp extension of $\id_{\A}$). Since
\[
    a\theta(1) = \theta(a) = \theta(1)a, ~ a \in \A,
\]
we see that $\theta(1) \in \A' \cap I(\A) = Z(I(\A))$. We claim that $\LIM_n (\theta(1) + 1/n)^{-1}\theta(x)$ exists for all $x \in \B$. Indeed, for all $x \in \B_+$, $\{(\theta(1) + 1/n)^{-1}\theta(x)\} \subseteq I(\A)_+$ is an increasing sequence bounded above by $\|x\|$. In particular, $\LIM_n (\theta(1) + 1/n)^{-1}\theta(1) = p$, where $p = \LP(\theta(1)) = \RP(\theta(1)) \in Z(I(\A))$. Now define a ucp map $\tilde{\theta}:\B \to I(\A)$ by the formula
\[
    \tilde{\theta}(x) = \LIM_n (\theta(1) + 1/n)^{-1}\theta(x) + p^\perp \Phi(x), ~ x \in \B.
\]
Since $\theta$ and $\Phi$ are $\A$-bimodular, so is $\tilde{\theta}$. Furthermore,
\[
    \theta(1)\tilde{\theta}(x) = p\theta(x), ~ x \in \B.
\]
But $p\theta(x) = \theta(x)$, $x \in \B$, since for all $x \in \B_{sa}$,
\[
    -\|x\| \leq x \leq \|x\| \implies -\|x\|\theta(1) \leq \theta(x) \leq \|x\|\theta(1)
    \implies p^\perp\theta(x) = 0.
\]
Thus
\[
    \theta(1)\tilde{\theta}(x) = \theta(x), ~ x \in \B.
\]
\end{proof}

\begin{theorem} \label{pseudo}
Let $(\A,G,\alpha)$ be a discrete $C^*$-dynamical system. Then $\A \subseteq \A \rtimes_r G$ admits a unique pseudo-expectation if and only if the action of $G$ on $\A$ is properly outer.
\end{theorem}

\begin{proof}
($\Rightarrow$) Suppose that $\alpha_g \in \Aut(\A)$ is not properly outer for some $g \neq e$. Then $\tilde{\alpha}_g \in \Aut(I(\A))$ is not properly outer (hence not freely acting), and so there exists $0 \neq v \in I(\A)$ such that $vx = \tilde{\alpha}_g(x)v$, $x \in I(\A)$. In particular, $va = \alpha_g(a)v$, $a \in \A$. Define a completely bounded map $\theta:\A \rtimes_r G \to I(\A)$ by the formula
\[
    \theta(x) = \bbE(x) + \bbE(xg^{-1})v, ~ x \in \A \rtimes_r G,
\]
where $\bbE:\A \rtimes_r G \to \A$ is the canonical conditional expectation. Note that $\theta(g) = v \neq 0$. Obviously $\theta$ is a left $\A$-bimodule map, since $\bbE$ is. It is also a right $\A$-bimodule map, since for all $x \in \A \rtimes_r G$ and all $a \in \A$, we have that
\begin{eqnarray*}
    \theta(xa)
    &=& \bbE(xa) + \bbE(xag^{-1})v\\
    &=& \bbE(xa) + \bbE(xg^{-1}gag^{-1})v\\
    &=& \bbE(xa) + \bbE(xg^{-1}\alpha_g(a))v\\
    &=& \bbE(x)a + \bbE(xg^{-1})\alpha_g(a)v\\
    &=& \bbE(x)a + \bbE(xg^{-1})va\\
    &=& (\bbE(x) + \bbE(xg^{-1})v)a\\
    &=& \theta(x)a.
\end{eqnarray*}
By \cite[Thm. 4.5]{Wittstock1981}, $\theta = (\theta_1 - \theta_2) + i(\theta_3 - \theta_4)$, where $\theta_j:\A \rtimes_r G \to I(\A)$ is a completely positive $\A$-bimodule map, $1 \leq j \leq 4$. Without loss of generality, $\theta_1(g) \neq 0$. By Lemma \ref{factorization}, there exists a pseudo-expectation $\tilde{\theta}_1:\A \rtimes_r G \to I(\A)$ for $\A \subseteq \A \rtimes_r G$ such that $\theta_1(x) = \theta_1(1)\tilde{\theta}_1(x)$, $x \in \A \rtimes_r G$. In particular, $\tilde{\theta}_1(g) \neq 0$, so that $\tilde{\theta}_1 \neq \bbE$.

($\Leftarrow$) Conversely, suppose that $\alpha_g \in \Aut(\A)$ is properly outer for all $g \neq e$. Then $\tilde{\alpha}_g \in \Aut(I(\A))$ is properly outer (hence freely acting) for all $g \neq e$. Let $\theta:\A \rtimes_r G \to I(\A)$ be a pseudo-expectation for $\A \subseteq \A \rtimes_r G$. For $g \in G$, we have that
\[
    gag^{-1} = \alpha_g(a) \implies ga = \alpha_g(a)g \implies \theta(g)a = \alpha_g(a)\theta(g), ~ a \in \A.
\]
By Lemma \ref{dependent}, we have that
\[
    \theta(g)t = \tilde{\alpha}_g(t)\theta(g), ~ t \in I(\A).
\]
Thus $\theta(g) = 0$ for all $g \neq e$. Hence $\theta = \bbE$.
\end{proof}

As pointed out to us by David Pitts, the proof of Theorem \ref{pseudo} can be repeated verbatim with $\A \rtimes_r G$ replaced by $\A \rtimes G$ and $\bbE:\A \rtimes_r G \to \A$ replaced by $\tilde{\bbE} = \bbE \circ \lambda:\A \rtimes G \to \A$. Thus we have:

\begin{corollary} \label{pseudo_full}
Let $(\A,G,\alpha)$ be a discrete $C^*$-dynamical system. Then $\A \subseteq \A \rtimes G$ (full crossed product) admits a unique pseudo-expectation if and only if the action of $G$ on $\A$ is properly outer.
\end{corollary}

\section{Applications}

\subsection{Special Inclusions}

In this section we specialize Theorems \ref{CE} and \ref{pseudo} and their corollaries to particular cases, namely $\A$ abelian and $\A$ simple. We begin with the case $\A$ abelian.

\begin{remark} \label{abelian}
If $\A$ is unital abelian $C^*$-algebra, then every $\alpha \in \Aut(\A)$ induces a homeomorphism $\hat{\alpha}:\widehat{\A} \to \widehat{\A}$ by the formula $\hat{\alpha}(\sigma) = \sigma \circ \alpha$, $\sigma \in \widehat{\A}$. In that case, the following are equivalent:
\begin{enumerate}
\item[i.] $\alpha$ is properly outer;
\item[ii.] $\alpha$ is freely acting;
\item[iii.] $\hat{\alpha}$ is topologically free, i.e., $\Fix(\hat{\alpha})^\circ = \emptyset$.
\end{enumerate}
\end{remark}

\begin{proof}
(i $\implies$ ii) True in general, not just the abelian case.

(ii $\implies$ i) Suppose $\alpha$ is freely acting. Let $\J \subseteq \A$ be an $\alpha$-invariant ideal such that $\alpha|_{\J}$ is quasi-inner. Then $\widetilde{\alpha|_{\J}} = \tilde{\alpha}|_{I(\J)}$ is inner, therefore the identity map. Hence $\alpha|_{\J}$ is the identity map. Now let $h \in \J$. For all $a \in \A$, we have that
\[
    ha = \alpha(ha) = \alpha(h)\alpha(a) = h\alpha(a) = \alpha(a)h.
\]
Thus $h = 0$. Since the choice of $h$ was arbitrary, $\J = 0$ and $\alpha$ is properly outer.

(ii $\iff$ iii) \cite[Thm. 1]{EnomotoTamaki1974}.
\end{proof}

\begin{corollary}
Let $(\A,G,\alpha)$ be a discrete $C^*$-dynamical system, with $\A$ abelian. Then the following are equivalent:
\begin{enumerate}
\item[i.] $\A \subseteq \A \rtimes_r G$ (or $\A \subseteq \A \rtimes G$) admits a unique pseudo-expectation;
\item[ii.] $\A \subseteq \A \rtimes_r G$ (or $\A \subseteq \A \rtimes G$) admits a unique conditional expectation;
\item[iii.] $\A \subseteq \A \rtimes_r G$ is a MASA;
\item[iv.] $G$ acts topologically freely on $\widehat{\A}$.
\end{enumerate}
\end{corollary}

(In particular, we recover \cite[Thm. 4.6]{PittsZarikian2015}.)\\

Now we consider the case $\A$ simple.

\begin{remark} \label{simple}
If $\A$ is simple unital $C^*$-algebra and $\alpha \in \Aut(\A)$, then the following are equivalent:
\begin{enumerate}
\item[i.] $\alpha$ is properly outer;
\item[ii.] $\alpha$ is freely acting;
\item[iii.] $\alpha$ is outer (i.e., not inner).
\end{enumerate}
\end{remark}

\begin{proof}
(i $\implies$ ii) True in general, not just the simple case.

(ii $\implies$ iii) True in general, not just the simple case. (Consider the contrapositive.)

(iii $\implies$ i) Suppose $\alpha$ is outer. By \cite[Thm. 3.6]{SaitoWright1984}, $\tilde{\alpha}$ is outer. Now $I(\A)$ is simple, and therefore a factor \cite[Prop. 4.15]{Hamana1979}. Thus $\tilde{\alpha}$ is properly outer, so that $\alpha$ is properly outer.
\end{proof}

\begin{corollary}
Let $(\A,G,\alpha)$ be a discrete $C^*$-dynamical system, with $\A$ simple. Then the following are equivalent:
\begin{enumerate}
\item[i.] $\A \subseteq \A \rtimes_r G$ (or $\A \subseteq \A \rtimes G$) admits a unique pseudo-expectation;
\item[ii.] $\A \subseteq \A \rtimes_r G$ (or $\A \subseteq \A \rtimes G$) admits a unique conditional expectation;
\item[iii.] The action of $G$ on $\A$ is outer.
\end{enumerate}
\end{corollary}

\subsection{Simplicity}

In this section, we use Theorem \ref{pseudo} and Corollary \ref{pseudo_full} to quickly reprove (and slightly strengthen) some known $C^*$-simplicity results for reduced crossed products.\\

In \cite[Thm. 3.1]{Kishimoto1981}, Kishimoto proves that if a discrete group $G$ acts outerly on a simple unital $C^*$-algebra $\A$ by $*$-automorphisms, then $\A \rtimes_r G$ is simple. It follows that $\A \rtimes_r H$ is simple for any subgroup $H \subseteq G$. Recently, Cameron and Smith obtained the beautiful result that every intermediate $C^*$-algebra $\A \subseteq \B \subseteq \A \rtimes_r G$ has this form \cite[Thm. 3.5]{CameronSmith2017}. Combining these statements gives:

\begin{theorem}[\cite{Kishimoto1981,CameronSmith2017}] \label{Kishimoto}
Let $G$ be a discrete group acting outerly on a simple unital $C^*$-algebra $\A$ by $*$-automorphisms. Then every intermediate $C^*$-algebra $\A \subseteq \B \subseteq \A \rtimes_r G$ is simple.
\end{theorem}

We present a quick proof which bypasses \cite{CameronSmith2017}.

\begin{proof}
By Remark \ref{simple}, the action of $G$ on $\A$ is properly outer, and so by Theorem \ref{pseudo} the inclusion $\A \subseteq \A \rtimes_r G$ has a unique pseudo-expectation, which is actually a faithful conditional expectation. By \cite[Thm. 3.5]{PittsZarikian2015}, the inclusion $\A \subseteq \A \rtimes _r G$ is \emph{hereditarily essential} (See the introduction for a reminder of what this means.) Now suppose $\A \subseteq \B \subseteq \A \rtimes_r G$ is an intermediate $C^*$-algebra and $\J \subseteq \B$ is a non-zero ideal. Then $\J \cap \A \neq 0$, which implies $\J \cap \A = \A$, which in turn implies $1 \in \A \subseteq \J$. Hence $\J = \B$, and $\B$ is simple.
\end{proof}

In \cite{ArchboldSpielberg1993}, Archbold and Spielberg generalize the notion of \emph{topological freeness} for automorphisms of abelian $C^*$-algebras (described in Remark \ref{abelian} above) to actions of discrete groups on general (non-abelian) $C^*$-algebras. Formally, topological freeness is stronger than proper outerness \cite[Prop. 1]{ArchboldSpielberg1993}, although they may be the same \cite[Rmk. 3.2]{KirchbergSierakowski2015}. One of their main results is the following:

\begin{theorem}[\cite{ArchboldSpielberg1993}, Thm. 1] \label{ArchboldSpielberg}
Let $G$ be a discrete group acting topologically freely on a unital $C^*$-algebra $\A$ by $*$-automorphisms. If $\J \subseteq \A \rtimes G$ is an ideal such that $\J \cap \A = 0$, then $\J \subseteq \ker(\lambda)$, where $\lambda:\A \rtimes G \to \A \rtimes_r G$ is the canonical $*$-homomorphism.
\end{theorem}

An immediate corollary is the following $C^*$-simplicity result:

\begin{corollary}[\cite{ArchboldSpielberg1993}] \label{ArchboldSpielbergCor}
Let $G$ be a discrete group acting topologically freely and minimally on a unital $C^*$-algebra $\A$ by $*$-automorphisms. Then $\A \rtimes_r G$ is simple.
\end{corollary}

(We say that $G$ acts \emph{minimally} on $\A$ if there are no non-trivial $G$-invariant ideals.)\\

We can economically prove Theorem \ref{ArchboldSpielberg} under the (apparently) weaker hypothesis that the action of $G$ on $\A$ is properly outer, thereby (seemingly) strengthening the result. Corollary \ref{ArchboldSpielbergCor} can be correspondingly strengthened.

\begin{proof}[Proof of strengthened Theorem \ref{ArchboldSpielberg}.]
Define a unital $*$-homomorphism $\pi:\A + \J \to \A: a + h \mapsto a$. By injectivity, $\pi$ extends to a pseudo-expectation $\theta:\A \rtimes G \to I(\A)$ for $\A \subseteq \A \rtimes G$. By Corollary \ref{pseudo_full} (which only requires proper outerness of the action of $G$ on $\A$), $\theta = \bbE \circ \lambda$. Thus
\[
    h \in \J
    \implies \bbE(\lambda(h)^*\lambda(h)) = \bbE(\lambda(h^*h)) = \theta(h^*h) = \pi(h^*h) = 0
    \implies \lambda(h) = 0.
\]
Hence $\J \subseteq \ker(\lambda)$.
\end{proof}

This inevitably leads us to wonder:

\begin{question}
Could it be that topological freeness and proper outerness are actually equivalent?
\end{question}

\subsection{Unique Conditional Expectation but Multiple Pseudo-Expectations}

In \cite[Ex. 4.4]{Zarikian2017}, we produce a $C^*$-inclusion $\A \subseteq \B$ with a unique conditional expectation, but infinitely many pseudo-expectations. In fact, $\B$ is abelian in our example. Unfortunately, the construction is a bit ad hoc. Also, the conditional expectation is not faithful. Now we can produce many such examples systematically. Indeed, if $\A$ is a unital $C^*$-algebra and $G$ is a discrete group acting freely but not properly outerly on $\A$ by $*$-automorphisms, then the $C^*$-inclusion $\A \subseteq \A \rtimes_r G$ admits a unique (faithful) conditional expectation, but infinitely many pseudo-expectations.

\end{document}